\documentclass{amsart}
\usepackage{amssymb} 
\usepackage{amsmath} 
\usepackage{amscd}
\usepackage{amsbsy}
\usepackage{comment}
\usepackage[matrix,arrow]{xy}
\usepackage{hyperref}
\usepackage{amsrefs}
\usepackage{tikz-cd}
\usepackage{listings}
\usepackage{xcolor}

\DeclareMathOperator{\Tr}{Trace}
\DeclareMathOperator{\Frob}{Frob}

\DeclareMathOperator{\PGL}{PGL}

\DeclareMathOperator{\Aut}{Aut}

\DeclareMathOperator{\GL}{GL}
\DeclareMathOperator{\SL}{SL}

\DeclareMathOperator{\Gal}{Gal}
\DeclareMathOperator{\norm}{Norm}

\DeclareMathOperator{\Coker}{Coker}

\newcommand{\Q}{{\mathbb Q}}
\newcommand{\Z}{{\mathbb Z}}
\newcommand{\C}{{\mathbb C}}

\newcommand{\F}{{\mathbb F}}
\newcommand{\cA}{\mathcal{A}}

\newcommand{\cN}{\mathcal{N}}

\newcommand{\cO}{\mathcal{O}}

\newcommand{\ff}{\mathfrak{f}}

\newcommand{\fq}{\mathfrak{q}}

\newcommand{\rhobar}{\overline{\rho}}

\newtheorem{theorem}{Theorem}[section]

\newtheorem{conjecture}[theorem]{Conjecture}
\newtheorem{corollary}[theorem]{Corollary}
\newtheorem{remark}[theorem]{Remark}
\newtheorem{definition}[theorem]{Definition}

\newcommand{\hooklongrightarrow}{\lhook\joinrel\longrightarrow}

\begin{document}
	\title[Serre's Modularity and Fermat over quadratic imaginary fields ]{On Serre's modularity conjecture and Fermat's equation over quadratic imaginary fields of class number one}

\begin{abstract}
 In the present article, we extend previous results of the author and we show that when $K$ is any quadratic imaginary field of class number one, Fermat's equation $a^p+b^p+c^p=0$ does not have integral coprime solutions $a,b,c \in K \setminus \{ 0 \}$ such that $2 \mid abc$ and $p \geq 19$ is prime. The results are conjectural upon the veracity of a natural generalisation of Serre's modularity conjecture. 
\end{abstract}

\author{George C. \c Turca\c s}
\address{Mathematics Institute\\
	University of Warwick\\
	Coventry\\
	United Kingdom}

\email{george.turcas@warwick.ac.uk}


\maketitle

\section{Introduction}\label{Sec1}

Let $K$ be a number field and denote by $\cO_K$ its ring of integers. The Fermat equation with prime exponent $p$ over $K$ is the equation
\begin{equation} \label{eq1} a^p+b^p+c^p=0, \hspace{1cm} a, b, c \in \cO_K.\end{equation} Wiles's extraordinary proof of Fermat's Last Theorem inspired mathematicians to attack \eqref{eq1} via Frey elliptic curves and modularity over $K$. Successful attempts had been carried out by Jarvis and Meekin \cite{jarvis2004fermat}, and Freitas and Siksek \cite{FreitasSiksek}, \cite{AFreitasSiksek}. They all rely on progress in modularity lifting over totally real fields due to work of Barnett-Lamb, Breuil, Diamond, Gee, Geraghty, Kisin, Skinner,
Taylor, Wiles and others. Modularity of elliptic curves over real quadratic fields was proved by Freitas, Le Hung and Siksek \cite{Freitas2015}. Recently, Derickx, Najman and Siksek \cite{Der19} proved that elliptic curves over totally real cubic fields are modular.

On the other hand, modularity of elliptic curves over number fields with complex embeddings is highly conjectural. For general number fields, \c Seng\" un and Siksek \cite{ASiksek} proved an asymptotic version of Fermat's Last Theorem, under the assumption of two standard, but very deep conjectures in the Langlands programme.  They proved that for a number field $K$ satisfying a precise $S$-unit equation criterion, there exists an ineffective constant $B_K$, depending only on the field $K$, such that for all primes $p>B_K$, the equation
$$a^p + b^p + c^p = 0,$$
does not have solutions in $K \setminus \{0\}$.

 Let us denote by $G_K=\Gal(\overline K/K)$ the absolute Galois group of $K$. 
Our results assume a version of Serre's modularity conjecture (see Conjecture \ref{serreconj}) for odd, absolutely irreducible, continuous $2$-dimensional mod $p$ representations of $G_K$ that are finite flat at every prime above $p$. 

In \cite{turcas2018}, the author proved the following theorem
\begin{theorem} \label{main} Let $K = \Q(\sqrt{-d})$, where $d \in \{1,2,7\}$. Assume Conjecture \ref{serreconj} holds for $K$. If $p \geq 5$ is a rational prime number, then the equation	
	\begin{equation} \label{maineq} a^p + b^p + c^p = 0 \end{equation}
	has no solutions $a,b,c \in K\setminus \{ 0\}$.
\end{theorem}
Using known results about solutions to the equations $a^3+b^3=c^3$ and $a^4+b^4=c^4$ over $\Q(i)$, one obtains the easy corollary (see \cite{turcas2018}*{Corollary 1.2})

\begin{corollary} \label{maincor} Assume Conjecture \ref{serreconj} holds for $\mathbb Q(i)$. Then, Fermat's Last Theorem holds over $\mathbb Q(i)$. In other words, for any integer $n \geq 3$, the equation
	$$a^n+b^n=c^n$$
	has no solution $a,b,c \in \mathbb Q(i) \setminus \{ 0\}$. 
\end{corollary}

In the present work, we consider Fermat's equation over the remaining six quadratic imaginary fields of class number one. Our main result is as follows.

\begin{theorem}
	\label{2inert} Let $K= \Q(\sqrt{-d})$, where $d \in \{3,11,19,43, 67,163\}$. Assume Conjecture \ref{serreconj} holds for $K$. For any prime $p \geq 19$, the Fermat equation 
	$$a^{p}+b^{p}+c^p =0,$$
	does not have solutions in coprime $a, b, c \in \mathcal O_K \setminus \{ 0\}$ such that $2 \mid abc$.
\end{theorem}

\begin{remark}
	When $d \neq 163$, the statement of the previous theorem is true for $p \geq 17$. We only have to take  $p \geq 19$ for $K=\mathbb Q(\sqrt{-163})$ due to the presence of $17$-torsion in the integral cohomology of the relevant locally symmetric space.
\end{remark}

For an overview of the difficulties that had to be overcome to prove Theorems \ref{main} and \ref{2inert}, we refer to \cite{turcas2018}*{Section 1}. To prove Theorem \ref{main}, we made essential use of the fact that there is a prime ideal of $\mathcal O_K$ of residue field $\F_2$. This turns out to be an essential ingredient for proving that certain mod $p$ Galois representations are absolutely irreducible, a hypothesis of Conjecture \ref{serreconj}. Indeed, the rational prime $2$ ramifies in $\Q(i)$ and $\Q(\sqrt{-2})$ and splits completely in $\Q(\sqrt{-7})$. The remaining quadratic imaginary fields of class number one are $\Q(\sqrt{-3})$, $\Q(\sqrt{-11})$, $\Q(\sqrt{-19})$, $\Q(\sqrt{-43})$, $\Q(\sqrt{-67})$ and $\Q(\sqrt{-163})$. The prime $2$ is inert in all six of them. In order to prove that certain mod $p$ Galois representations are absolutely irreducible, we need to add the assumption that $2 \mid abc$.

 It would have been extremely satisfying to obtain a statement of Theorem \ref{2inert} analogous to Theorem \ref{main}, showing that for $p$ larger than a fixed constant depending only on $K$, the equation \eqref{maineq} has no solutions $a,b,c \in K \setminus \{ 0\}$ for all the six fields $K$. Such a statement is not even true when $K=\Q(\sqrt{-3})$. One can see this by considering the triple $(a,b,c)=(1, \omega, \omega^2)$, where $\omega \in \Q(\sqrt{-3})$ is a primitive third root of unity, a solution to \eqref{maineq} for every prime exponent $p \geq 5$. However, in Section 5 we show that another folklore conjecture due to Serre, which concerns the surjectivity of mod $p$ Galois representations, together with Conjecture \ref{serreconj} imply that when $K$ is imaginary quadratic of class number one, there is a constant $B_K$, depending only on $K$, such that for every prime $p > B_K$ if $(a,b,c) \in K \setminus \{0 \}$ is a solution to \eqref{main}, then $K = \Q(\sqrt{-3})$ and $(a,b,c)=(1, \omega, \omega^2)$, up to permutation.

\subsection*{Acknowledgements} We are indebted to the referee for his careful reading of this paper and for suggesting numerous improvements. The author is grateful to his advisor Samir Siksek for suggesting the problem and for his great support. It is also a pleasure to thank John Cremona, Toby Gee, Aurel Page and Haluk \c Seng\" un for useful discussions. The author is supported by EPSRC Programme Grant EP/K034383/1 LMF: L-Functions and Modular Forms.

\section{Serre's modularity conjecture}

Throughout this section $K$ is a quadratic imaginary field of class number one and $p \geq 5$ a rational prime that does not ramify in $K$. For notations, relevant definitions and a detailed discussion concerning complex and mod $p$ eigenforms for $\GL_2$ over $K$, we refer the reader to \cite{turcas2018}*{Section 2}. 

Let $\cN \subseteq \cO_K$ be an ideal and consider the locally symmetric space
$$Y_0(\cN) = \Gamma_0(\cN) \backslash \mathcal H_3,$$
where $\Gamma_0(\cN)= \left\{  \left( \begin{array}{cc} a & b \\ c & d \end{array} \right) \in \GL_2(\mathcal O_K) \mid c \in \cN\right\}$ is the usual congruence subgroup for the modular group $\GL_2(\mathcal O_K)$ and $\mathcal H_3$ is the hyperbolic $3$-space. It is explained in detail in loc. cit. that, in our setting, we can use the following definition.

\begin{definition} By a complex (or mod $p$) eigenform of level $\cN$ we understand a 
	cohomology class $c \in H^1(Y_0(\mathcal N), \C)$ (or $H^1(Y_0(\mathcal N), \overline \F_p)$) that is a simultaneous eigenvector for all the Hecke operators $T_{\pi}$, where $(\pi) \subset O_K$ is a prime ideal coprime to $\cN$ (or $p\cN$ respectively).
\end{definition}

The eigenforms in these algebras are sometimes called Bianchi modular forms (or mod $p$ Bianchi modular forms) in the literature. Via the Eichler-Shimura-Harder isomorphism \cite{Harder87}*{Section 3.1}, Bianchi modular forms can be analytically interpreted as vector valued real-analytic functions on the hyperbolic 3-space (see \cite{crewhit} for more details). Unlike the classical situation in which $K=\Q$, when $K$ is quadratic imaginary not all the mod $p$ eigenforms lift to complex ones. The obstruction for such lifting is the possible presence of $p$-torsion in the second integral cohomology $H^2(Y_0(\cN), \Z_{(p)})$, where $\Z_{(p)}$ is the set of rational numbers with denominators prime to $p$. In the proof of Theorem \ref{2inert}, we prove that for the prime numbers $p$ under consideration, the aforementioned cohomology group does not have $p$-torsion. Appealing to a result of Ash and Stevens \cite{ash1986}, we show that mod $p$ eigenforms lift to complex ones.

We will be using a special case of Serre's modularity conjecture over number fields. In his landmark paper \cite{Serre}, Serre  conjectured that all absolutely irreducible, odd mod $p$ Galois representation of $G_{\mathbb Q}=\Gal(\overline{\mathbb Q}/ \mathbb Q)$ \textit{arise} from a cuspidal eigenform $f$. Here $f$ is a classical cuspidal modular form. In the same article, Serre gave a recipe for the level $N$ and the weight $k$ of the sought after eigenform. This conjecture was proved by Khare and Wintenberger \cite{Khare2009}.

We are going to state a conjecture which concerns mod $p$ representations of $G_K$, the absolute Galois group of the quadratic imaginary $K$.

\begin{conjecture}[compare to \cite{ASiksek}*{Conjecture 3.1}] \label{serreconj} Let $p>3$ be a prime and $E$ an elliptic curve defined over $K$. Suppose that the mod $p$ Galois representation $\overline \rho_E : G_K \to \GL_2(\overline{\mathbb F}_p)$, induced by the action of $G_K$ on the $p$-torsion of E, is absolutely irreducible, continuous with Serre conductor $\mathcal N$ (prime-to-$p$ part of its Artin conductor). Assume that $\overline \rho_E|_{G_{K_{\mathfrak p}}}$ arises from a finite-flat group scheme over $\mathcal O_{K_{\mathfrak p}}$ for every prime $\mathfrak p|p$. Then there is a mod $p$ eigenform $c \in H^1(Y_0(\cN),\overline{ \F}_p)$ such that for all prime ideals $(\pi) \subseteq \mathcal O_K$, coprime to $p \cN$
	$$T_{\pi}(c) = \Tr ( \overline \rho_E( \Frob_{ (\pi) }) ) \cdot c.$$
\end{conjecture}

\medskip

Although it is conjecturally easy to predict the level $\mathcal N$ of such an eigenform, doing the same thing for the weight can be very difficult. A quite involved general weight recipe for $\GL_2$ over number fields was given by Gee, Herzig and Savitt in \cite{GeeHerzSav}. We just mention that this recipe depends on the restriction $\overline{\rho}|_{I_{\mathfrak p}}$ to the inertia subgroups for the primes $\mathfrak p \subset \mathcal O_K$ above $p$.  We only considered very special representations $\overline{\rho}$ (that are finite flat at $\mathfrak p|p$), for which Serre's original weight recipe applies and predicts the trivial weight \cite{Serre}. This is why we end up with classes in $H^1(Y_0(\mathfrak N),  \overline{\mathbb F}_p)$, the trivial weight meaning that we get $\overline{\mathbb F}_p$ as coefficient module. 

As opposed to \cite{ASiksek}*{Conjecture 3.1}, in Conjecture \ref{serreconj} we allow for the prime $p$ to ramify in $K$. The restriction in loc. cit. originates in the weight recipe given by Buzzard, Diamond and Jarvis \cite{BDJ10}, where the authors had to assume that the prime $p$ is unramified in $K$. This restriction was since then removed in a series of papers of Gee et al. \cite{GeeLiuSav,GeeHerzSav}. Due to technicalities arising in the weight recipe when $p$ is ramified in $K$, we had to assume that our Galois representation is the reduction of one coming from an elliptic curve. Indeed, by \cite{GeeLiuSav}*{Section 4.1.2} the hypothesis on $\rhobar_{E}$ imply that all $a_i=b_i=0$ in loc. cit. and this predicts a trivial weight if and only if the representation has a crystaline lift with all pairs of labelled Hodge-Tate weights equal to $\{0,1 \}$. In the hypothesis of the conjecture formulated above, $\rhobar_E$ is the mod $p$ reduction of a $p$-adic representation coming from an elliptic curve, therefore such a lift already exists. The recipes for the weights in \cite{BDJ10} and \cite{GeeLiuSav} were given for totally real fields but as the problem of weights is a local issue, these (conjecturally) apply to any number field.

\medskip

\noindent\textbf{Remark.}  When stated for more general number fields, this conjecture restricts to odd representations. A representation is odd if the determinant of every complex conjugation is $-1$, but since our $K$ is totally complex $G_K$ does not contain any complex conjugations and we will regard every mod $p$ representation of $G_K$ automatically as odd. 

\section{Fermat equation with exponent $p$ and the Frey curve}
\label{FermatpFrey}

Let us fix some notation:

$p$ - a rational prime number;

$K= \Q(\sqrt{-d})$ where $d$ is one of  $3$, $11$, $19$, $43$, $67$ or $163$;

$\mathcal O_K$ - the ring of integers of $K$;

$\mathfrak q = (2 \cO_K)$ - the prime ideal of residue field $\F_{4}$ that lies above $2$. 

\vspace{0.5cm}

By the Fermat equation with exponent $p$ over $K$, we mean
\begin{equation} \label{fermeq}
a^p + b^p + c^p=0, \hspace{1cm} a,b,c \in \mathcal O_K.
\end{equation}

We say that a solution $(a,b,c) \in \mathcal O_K^3$ to the equation above is trivial if $abc = 0$ and non-trivial otherwise. We shall henceforth assume that $p \geq 19$ when $d=163$ and $p \geq 17$ otherwise.

 Let $(a,b,c) \in \cO_K^3$ be a non-trivial solution to \eqref{fermeq} such that $a,b,c$ are coprime. One can always assume that $a,b,c$ are coprime, since the class number of $K$ is one. Associated to $(a,b,c)$ is the Frey curve
\begin{equation} \label{freycurve}
E=E_{a,b,c}: Y^2 = X(X-a^p)(X+b^p).
\end{equation}

Write $\overline{\rho}=\overline{\rho}_{E,p}$ for the residual Galois representation
$$\overline{\rho}_{E,p}: G_K \to \Aut(E[p]) \cong \GL_2(\mathbb F_p )$$
induced by the action of $G_K$ on the $p$-torsion of $E[p]$. The following are easy to deduce proprieties of $E$ and $\rhobar$ which are proved in multiple papers concerning Fermat's equation over number fields. For a precise reference, see \cite{turcas2018}*{Lemma 3.1}:
\begin{itemize}
	\item the elliptic curve $E$ is semistable away from $\mathfrak q= (2 \cO_K)$;
	\item the Galois representation $\rhobar$ is unramified at the primes not dividing $2p$, finite flat at every prime $\mathfrak p$ of $\cO_K$ that lies above $p$ and $\det(\rhobar)=\chi_p$, the mod $p$ cyclotomic character.
\end{itemize}

From now on, we add the assumption that $2 \mid abc$. It follows from \cite{AFreitasSiksek}*{Lemma 4.2} that $E$ has potentially multiplicative reduction at $\mathfrak q$. We denote by $N_E$ the conductor of $E$ and we will apply \cite{AFreitasSiksek}*{Lemma 4.4} to determine $v_{\mathfrak q}(N_E)$, the valuation of $N_E$ at $\mathfrak q$. Let us demystify the quantities introduced in the aforementioned lemma. The ideal $\mathfrak b$ is equal to $\mathfrak q^3$ and has norm $64$.  The group of units of the quotient $\cO_K/ \mathfrak b$ is isomorphic to $(\Z/2\Z)^2 \oplus \Z/12\Z $ and therefore the co-domain of $\Phi : \mathcal O_K^{*} \to (\mathcal O_K/ \mathfrak b)^*/(\mathcal O_K / \mathfrak b)^{*2} $ is isomorphic to $(\Z/2\Z)^3$. For every considered $K$, the image of $\Phi$ is isomorphic with $\Z/2\Z$, so $\Coker(\Phi) \cong (\Z/2\Z)^2$.

We present, for each of the considered six fields $K=\Q(\sqrt{-d})$, a complete list of representatives of this cokernel and the maximal value for the exponent of $\mathfrak q$ in $N_E$. As it can be seen in the table below, if $K= \Q(\sqrt{-d})$ is one of the aforementioned fields and $2 \mid abc$, by \cite{AFreitasSiksek}*{Lemma 4.4} we can scale the triple $(a,b,c)$ by a unit so that the valuation of the conductor $N_E$ of the Frey curve at $\mathfrak q = 2 \cO_K$ is at most $4$.

\begin{table}[h]	
	\caption{Computations associated to Lemma 4.4 in \cite{AFreitasSiksek}}
	\begin{center}
		\label{corepz}
		\begin{tabular}{|c|c|c|}
			\hline
			$d$     & \textbf{Reps. $\lambda_1, \lambda_2, \lambda_3, \lambda_4 \in \mathcal O_K$ of }$\Coker(\Phi)$ &  $v_{\mathfrak q}(N_E)$ \\ \hline
			$3$   &    $1, \, \frac{-1+3\sqrt{-3}}{2}, \,  3+2\sqrt{-3}, \, \frac{3-\sqrt{-3}}{2}$       &           $4$  \\ \hline
			$11$  &  $1, \,  \frac{-1+\sqrt{-11}}{2}, \, -1+2\sqrt{-11}, \,  \frac{-5-3\sqrt{-11}}{2}  $         &            $4$   \\ \hline
			$19$  &  $1, \, \frac{1+ 3\sqrt{-19}}{2}, \,   3+ 2\sqrt{-19}, \, \frac{9+ 3 \sqrt{-19}}{2}$         &           $4$    \\ \hline
			$43$ & $1, \,  \frac{-7- \sqrt{-43}}{2}, \,  -1+ 2\sqrt{-43}, \,   \frac{-3+3 \sqrt{-43}}{2}$          &              $4$    \\ \hline
			$67$ &  $1, \, \frac{1+ 3\sqrt{-67}}{2},\,  1+2\sqrt{-67}, \,  \frac{-9-3\sqrt{-67}}{2} $         &          $4$       \\ \hline
			$163$&  $1, \, \frac{1+3\sqrt{-163}}{2}, \,  1+2\sqrt{-163}, \, \frac{-9-3\sqrt{-163}}{2}  $         &          $4$          \\ \hline
		\end{tabular}
	\end{center}
\end{table}

For applying Conjecture \ref{serreconj} to the Galois representation $\overline{\rho}= \overline{\rho}_{E,p}$, we have to prove that it is absolutely irreducible. 

\begin{theorem} \label{IsIrred} Let $K$ be one of the six fields above and consider a prime $p \geq 17$. If $(a,b,c) \in \mathcal O_K^{3}$ with $2 \mid \norm(abc)$ is a non-trivial solution to \eqref{fermeq} such that $a$, $b$ and $c$ are coprime, then $\rhobar_{E,p}$ is absolutely irreducible. 
\end{theorem}

\begin{proof}
	
	We claim that absolute irreducibility follows from irreducibility. Suppose that $\rhobar_{E,p}$ is irreducible. 
	Let $\mathfrak q$ be a prime of $K$ above $2$ that divides the product $abc$. Recall that $E$ has potentially multiplicative reduction at $\mathfrak q$. 
	
	Using the identity $j(E)=\frac{2^4(b^{2p}-a^pc^p)}{(abc)^{2p}}$, we see that $v_{\mathfrak q}(j(E))=4-2pv_{\mathfrak q}(abc)$ is negative and $p \nmid v_{\mathfrak q}(j(E))$. It follows from the theory of Tate curves \cite{silver2}*{Proposition 6.1} that there is an element $\sigma \in I_{\mathfrak q} \subseteq G_K$ that acts on $E[p]$ via a matrix of the form $\left( \begin{array}{cc} 1 & 1 \\ 0 & 1 \end{array} \right)$. The image $\rhobar_{E,p}(G_K)$ is therefore an irreducible subgroup of $\GL_2(\F_p)$ which contains en element of order $p$ and the classification \cite{serre72}*{Proposition 15} of maximal subgroups of $\GL_2(\F_p)$ implies that $\SL_2(\F_p) \subseteq \rhobar_{E,p}(G_K)$. As $\SL_2(\mathbb F_p)$ is an absolutely irreducible subgroup of $\GL_2(\mathbb F_p)$, the claim is proved. If $p$ does not ramify in $K$, we can say even more, namely that $\rhobar_{E,p}$ is surjective. This is a consequence of the fact that $\det(\rhobar_{E,p})$ is the mod $p$ cyclotomic character, which is surjective when $K \cap \Q(\zeta_p) = \Q$.
	
	 It remains to prove that $\rhobar_{E,p}$ is irreducible. If we suppose the contrary, we can write
	\begin{equation} \label{rhored} \rhobar_{E,p} \sim \left( \begin{array}{cc} \theta & * \\ 0 & \theta'  \end{array} \right) \end{equation}
	where $\theta$ and $\theta'$ are characters $G_K \to \mathbb F_p^{*}$ and $\theta \theta'= \chi_p$, the mod $p$ cyclotomic character given by the action of $G_K$ on the group $\mu_p$ of $p$-th roots of unity. Let us denote by $\cN_{\theta}, \cN_{\theta'}$ the conductors of $\theta, \theta'$ respectively.
	These characters are unramified away from $p$ and $2\mathcal O_K = \mathfrak q$, the only prime of additive reduction for $E$ (see \cite{KrausQ}*{Lemma 1}).
	
Moreover, we saw above that $\mathfrak q$ is a prime of potentially multiplicative reduction for $E$. Write $D_{\mathfrak q} \subseteq G_K$ for the decomposition subgroup at $\mathfrak q$.  The restriction $\rhobar_{E,p}$ to $D_{\mathfrak q}$ is, up to semi-simplification, equal to $\phi \oplus \phi \cdot \chi_p$, where $\phi$ is at worst a quadratic character (see \cite{silver2}*{Theorem V.5.3}). In particular, both $\theta^2$ and $\theta'^2$ are unramified at $\mathfrak q$.
	
	\textbf{(i)} We will first assume that $p$ is coprime to either $\cN_{\theta}$ or $\cN_{\theta'}$. Since the conductor of an elliptic curve is isogeny invariant, by eventually replacing $E$ with the $p$-isogenous curve $E/ \langle \theta \rangle$ we can assume that $p$ is coprime to $\cN_{\theta}$. This implies that $\theta$ is unramified away from $\mathfrak q$. By the above, we infer that $\theta^2$ is everywhere unramified. The crucial fact that $K$ has class number $1$ allows us to deduce that $\theta^2$ is the trivial character. Now, observe that either $E$ or its twist by the quadratic character $\theta$ has a point of order $p$ defined over $K$. The former instance happens precisely when $\theta$ is trivial itself and the latter when $\theta$ is quadratic.
	
	The possible prime torsion of elliptic curves over quadratic fields have been determined by Kamienny and his result \cite{Kamienny1992}*{Theorem 3.1} implies that $p \leq 13$, a contradiction. 
	
	\textbf{(ii)} Suppose that $p$ is not coprime with $\cN_{\theta}$ nor with $\cN_{\theta'}$. 
	
	We now show that under the assumption that $\rhobar_{E,p}$ is reducible, the prime $p$ does not ramify in $K$. Suppose it does  and let $\mathfrak p$ be the unique prime ideal of $\mathcal O_K$ such that $p \mathcal O_K=\mathfrak p^2$. Recall that the Frey curve $E$ has semistable reduction at $\mathfrak p$. We therefore see from \cite{FreitasSiksek}*{Proposition 6.1 (ii)} that if $\mathfrak p$ is a prime of good ordinary, or multiplicative reduction then
	$$\rhobar_{E,p}|_{I_{\mathfrak p}} \sim \left( \begin{array}{cc} \chi_p & * \\ 0 & 1 \end{array} \right).$$
	This would imply that one of the diagonal characters is not ramified at $\mathfrak p$ and we showed in \textbf{(i)} that this is not possible. That leaves us with the possibility for $\mathfrak p$ to be a prime of good supersingular reduction. In this situation \cite{FreitasSiksek}*{Proposition 6.1.} asserts that either
	\begin{equation}\label{lev1and2} \rhobar_{E,p}|_{I_{\mathfrak p}} \sim \left( \begin{array}{cc} \psi_2^2 & 0 \\ 0 & \psi_2^{2p} \end{array} \right) \text{ or } \rhobar_{E,p}|_{I_{\mathfrak p}} \sim \left( \begin{array}{cc} \psi_1 & 0 \\ 0 & \psi_1 \end{array} \right), \end{equation}
	where $\psi_1: I_{\mathfrak p} \to \F_p^{*}$ and $\psi_2 : I_{\mathfrak p} \to \F_{p^2}^*$ are the level $1$ and respectively $2$  fundamental characters defined in \cite{serre72}. The first possibility in \eqref{lev1and2} implies that $\theta|_{I_{\mathfrak p}}= \psi_2^2$ or $\theta|_{\mathfrak p} = \psi_2^{2p}$, which is impossible since $\psi_2^2$ and $\psi_2^{2p}$ are not $\F_p$-valued. Hence the restrictions of both $\theta$ and $\theta'$ to $I_{\mathfrak p}$ coincide with $\psi_1$.
	
	Recall that $\theta$ and $\theta'$ are unramified outside $2\mathcal O_K=\mathfrak q$ and $\mathfrak p$, so their conductors $\cN_{\theta'}, \cN_{\theta'}$ are supported on these two primes. Define $\varepsilon : G_K \to \F_p^{*}$ by $$\varepsilon = \theta/ \theta' = \theta^2/\chi_p.$$
	Since the restrictions of $\theta, \theta'$ to $I_{\mathfrak p}$ coincide, the character $\epsilon$ is unramified at $\mathfrak p$. The latter is also unramified away from $\mathfrak q$, because $\theta$ and $\theta'$ are so. Its conductor $\cN_{\varepsilon}$ is then a power of $\mathfrak q$. When restricted to the inertia subgroup of $\mathfrak q$ the cyclotomic character $\chi_p$ is trivial, hence $\varepsilon|_{I_{\mathfrak q}} = \theta^2|_{I_{\mathfrak q}}$. We remarked at the start of this proof that $\theta^2$ is unramified at $\mathfrak q$, so $\varepsilon$ is everywhere unramified. Again, from the fact that $K$ has class number one we derive that $\varepsilon$ is trivial.
	
	Let $\sigma_{\mathfrak q}$ be a Frobenius element of $\mathfrak q$. Since $\mathfrak q$ is a prime of potentially multiplicative reduction, it is know (see for instance \cite{ASiksek}*{Lemma 6.3}) that the possible pairs of eigenvalues of $\rhobar_{E,p}(\sigma_{\mathfrak q})$ are $(1, \norm(\mathfrak q))$ or $(-1, - \norm(\mathfrak q))$. We therefore get
	$$ 1 = \varepsilon(\sigma_{\mathfrak q})=  \theta(\sigma_{\mathfrak q})/ \theta'(\sigma_{\mathfrak q})  \equiv \norm(\mathfrak q)^{\pm 1} \pmod{p},$$
	so $p \mid \norm(\mathfrak q)-1 =3 $, which contradicts the hypothesis on $p$.
	
	We proved that $p$ does not ramify in $K$. If $p$ is inert, we can apply \cite{KrausQ}*{Lemme 1} to deduce that at least one of $\theta$ or $\theta'$ does not ramify at $p\mathcal O_K$, which puts us again in case \textbf{(i)}.
	
	The only possibility remaining is that $p$ splits in $K$. Let $\mathfrak p_1, \mathfrak p_2$ be the two ideals of $\mathcal O_K$ such that $p \mathcal O_K= \mathfrak p_1 \mathfrak p_2$. These primes are both of semistable reduction for $E$ so by Lemme 1 in loc. cit., swapping $\theta$ and $\theta'$ if necessary, we can suppose that $\mathfrak p_1 \mid \cN_{\theta}$, $\mathfrak p_1 \nmid \cN_{\theta'}$ and $\mathfrak p_2 \mid \cN_{\theta'}$, $\mathfrak p_2 \nmid \mathcal N_{\theta}$. The primes $\mathfrak p_1, \mathfrak p_2$ are unramified so it follows from \cite{serre72}*{Proposition 12} that $E$ has good ordinary or multiplicative reduction at these primes and that $\theta|_{I_{\mathfrak p_1}} = \chi_p|_{I_{\mathfrak p_1}}$ and $\theta'|_{I_{\mathfrak p_2}}= \chi_p|_{I_{\mathfrak p_2}}$.
	
It follows that the character $\theta^2$ is unramified everywhere except $\mathfrak p_1$, because the only bad place $\mathfrak q$ of $E$ is of potential multiplicative reduction, as explained at the beginning of the proof. Using Lemma 4.3 in \cite{turcas2018} with $\alpha = 2 \in K$, it follows that
	$$\theta^2(\sigma_{\mathfrak q}) \equiv \norm_{K_{\mathfrak p_1}/\Q_p}\left( \iota_{\mathfrak p_1}(2) \right)^2 \pmod{p},$$
	where $\sigma_{\mathfrak q}$ is the Frobenius element at $\mathfrak q$.
	
	Appealing to Lemma 6.3 in \cite{ASiksek} again, we derive that $\theta^2(\sigma_{\mathfrak q}), \theta'^2(\sigma_{\mathfrak q})$ are congruent (up to reordering) to $1$ and $\norm^2(\mathfrak q)$ modulo $p$. After replacing $E$ by the isogenous curve $E/ \langle \theta \rangle$, we can assume that $\theta^2(\sigma_{\mathfrak q}) \equiv 1 \pmod{p}$. We have that $$\norm_{K_{\mathfrak p_1}/\Q_p}\left( \iota_{\mathfrak p_1}(2) \right)^2-1 =3,$$ so $p \mid 3$, a contradiction.
\end{proof}
Let $K$ be a general number field and $E$ an elliptic curve defined over $K$ such that
		it is semistable at all primes $\mathfrak p \subset \mathcal O_K$ above $p$ and 
	has potentially multiplicative reduction at a prime $\mathfrak q \neq \mathfrak p$.
 There exists an explicit constant $B_{K, \mathfrak q}$ such that for all $p > B_{K, \mathfrak q}$ the mod $p$ Galois representation arising by the action of $G_K$ on the $p$-torsion points of $E$ is absolutely irreducible (see \cite{ASiksek}*{Proposition 6.1 and Corollary 6.2}). In Theorem \ref{IsIrred} we find the smallest possible values for $B_{\mathfrak q, K}$ in the case when $K$ is quadratic imaginary of class number one and $\mathfrak q$ is a prime above $2$.

\section{The proof of Theorem \ref{2inert}}
\label{proof2inert}

We proceed by contradiction and, assuming there is such a solution, we first scale it such as in Table \ref{corepz} and then construct the Frey curve $$E = E_{a,b,c} : Y^2 = X(X-a^p)(X+b^p).$$
The next step in our approach is to show the mod $p$ Galois representation $\rhobar_{E,p}$ satisfies the hypothesis of Serre's conjecture. Absolute irreducibility was proved in Theorem \ref{IsIrred} and all the other hypothesis follow easily from the discussion at the beginning of the previous section. The Serre conductor $\mathcal N$ of $\rhobar_{E,p}$ is supported only on the prime $\mathfrak q = (2\mathcal O_K)$ and from Table \ref{corepz} we know that $\mathcal N$ is a divisor of $\mathfrak q^4$.

Conjecture \ref{serreconj} predicts the existence of a mod $p$ Bianchi modular form $c \in H^{1}(Y_0(\mathcal N), \overline{\F}_p)$ such that for every prime ideal $(\pi) \subset \mathcal O_K$, coprime to $p \mathcal N$ we have
$$T_{\pi}(c)= \Tr(\rhobar_{E,p}(\Frob_{(\pi)})) \cdot c.$$

The trace elements $\Tr(\rhobar_{E,p}(\Frob_{(\pi)}))$ lie in $\F_p$, therefore $c \in H^{1}(Y_0(\mathcal N), \F_p)$.

We  fix an embedding from $\overline{\Q} \hooklongrightarrow \mathbb C$.
Unlike the classical situation in which $K= \Q$, when $K$ is a general number field not all mod $p$ eigenforms lift to complex ones. To explain this, let us denote by $\Z_{(p)}$ the ring of rational numbers with denominators prime to $p$.  Consider the following short exact sequence given by multiplication-by-$p$

\begin{center}
	\begin{tikzcd}	
		0 \arrow[r] & \mathbb Z_{(p)} \arrow[r, "\times p"]  & \mathbb Z_{(p)} \arrow[r]  & \mathbb F_p \arrow[r] & 0
	\end{tikzcd}.
\end{center}
This gives rise to a long exact sequence on cohomology
\begin{center}
	\begin{tikzcd}
		\dots H^1(Y_0(\mathfrak N), \mathbb Z_{(p)}) \arrow[r,"\times p"] & H^1(Y_0(\mathfrak N), \mathbb Z_{(p)}) \arrow[r] \arrow[d, phantom, ""{coordinate, name=Z}] & H^1(Y_0(\mathfrak N), \mathbb F_p) \arrow[dll,
		"\delta",
		rounded corners,
		to path={ -- ([xshift=2ex]\tikztostart.east)
			|- (Z) [near end]\tikztonodes
			-| ([xshift=-2ex]\tikztotarget.west)
			-- (\tikztotarget)}] \\ H^{2}(Y_0(\mathfrak N), \mathbb Z_{(p)}) \arrow[r] & \dots
	\end{tikzcd}
\end{center}
from which we can extract the short exact sequence
\begin{center} 
	\begin{tikzcd}[column sep = small]
		0 \arrow[r] & H^1(Y_0(\mathfrak N), \mathbb Z_{(p)}) \otimes \mathbb F_p \arrow[r] & H^1(Y_0(\mathfrak N), \mathbb F_p) \arrow[r, "\delta"] & H^{2}(Y_0(\mathfrak N), \mathbb Z_{(p)})[p] \arrow[r] & 0
	\end{tikzcd}.	
	
\end{center}

In the above, the presence of $p$-torsion in $H^{2}(Y_0(\mathfrak N), \mathbb Z_{(p)})$ is the obstruction to surjectivity for the map $H^1(Y_0(\mathfrak N), \mathbb Z_{(p)}) \otimes \F_p \to H^1(Y_0(\mathfrak N), \mathbb F_p)$. If there is only trivial such torsion, then any Hecke eigenvector $\overline c$ in $H^{1}(Y_0(\mathfrak N), \mathbb F_p)$ comes from such an eigenvector in $H^{1}(Y_0(\mathfrak N), \mathbb Z_{(p)}) \otimes \F_p$. Using a lifting lemma of Ash and Stevens \cite{ash1986}*{Proposition 1.2.2}, we deduce that there are
\begin{enumerate}
	\item a finite integral extension $R$ of $\Z_{(p)}$
	\item  a prime $\mathfrak p$ of $R$ above $p$ and
	\item a Hecke eigenvector $c$ in $H^{1}(Y_0(\mathfrak N), R)$
\end{enumerate}
such that the Hecke eigenvalues of $c$ reduced modulo $\mathfrak p$ are equal to the ones of $\overline{c}$. Using our fixed embedding $\overline{\Q} \hookrightarrow \C$ we can regard $c$ as a class in $H^{1}(Y_0(\cN), \C)$, which implies the existence of our sought after complex eigenform. 

\medskip

\noindent \textbf{Remark.} We observe that in the paragraph above, $\overline{c}$ is not necessarily the reduction of $c$. The result that we cite only states that \textit{a system of eigenvalues} occurring in $\F_p$ may, after finite base extension, be lifted to a system occurring in $\Z_{(p)}$. The interested reader should consult \cite{ash1986}*{Section 1.2} for a more illuminating discussion. 

\medskip

  One sees that $H^2(Y_0(\mathcal N), \mathbb Z_{(p)})$ and $H^2(Y_0(\mathcal N), \mathbb Z)$ have the same $p$-torsion. As discussed in \cite{ash1986}*{page 202}, if the least common multiple of the orders of elements of finite order in $\Gamma_0(\mathcal N)$ is invertible in the coefficients module, then simplicial cohomology and group cohomology are the same. In our case, $\mathcal N = \mathfrak q^4$ and all the elements of finite order in $\Gamma_0(\mathcal N)$ have orders dividing $6$. Therefore,  $H^2(Y_0(\mathcal N), \mathbb Z[\frac{1}{6}])$ and $H^2(\Gamma_0(\mathcal N), \mathbb Z[\frac 1 6])$ are isomorphic as Hecke modules.

Lefschetz duality for cohomology with compact support \cite{sentors}*{Section 2} gives a relation between the first homology and the second cohomology
$H_{1}(\Gamma_0(\mathcal N), \mathbb Z\left[\frac{1}{6}\right]) \cong H^2(\Gamma_0(\mathcal N) , \mathbb Z \left[\frac{1}{6}\right])$. It is also known that the abelianization $\Gamma_0(\mathcal N)^{ab} \cong H_{1}(\Gamma_0(\mathcal N), \mathbb Z)$ and therefore, for primes $p >3$, if the group $H^2(Y_0(\mathcal N) , \mathbb Z)$ has a $p$-torsion element, then $\Gamma_0(\mathcal N)^{ab}$ will have a $p$-torsion as well. We compute the abelianizations $\Gamma_0(\cN)^{ab}$ using an algorithm of Haluk \c Seng\" un \cite{sentors}. The \textbf{Magma} implementation of this algorithm was kindly provided to us by its author. The algorithm requires as input presentations for $\PGL_2(\cO_K)$, which we compute using a program of Page \cite{Pag15}. The relevant \textbf{Magma} files can be found at
\begin{center}
	\url{https://warwick.ac.uk/fac/sci/maths/people/staff/turcas/fermatprog}.
\end{center}

We record the primes $l$ that appear as orders of torsion elements in $\Gamma_0(\cN)^{ab}$, for each of the six number fields in Table \ref{primetorsion2}.

\begin{center}
	\begin{table}[h]
		\centering
		\caption{prime torsion in $\Gamma_0(\mathcal N)^{ab}$}
		\label{primetorsion2}
		\begin{tabular}{|c|c|l|}
			\hline
			\textbf{Number field}         & \textbf{Level $\mathcal N$} & \textbf{primes $l$ such that $\Gamma_0(\mathcal N)^{ab}[l] \neq 0$ } \\ \hline
			$\mathbb Q(\sqrt{-3})$&      $(2\mathcal O_K)^4$      &  $2,3$          \\ \hline
			$\mathbb Q(\sqrt{-11})$&  $(2 \mathcal O_K)^4$         & $2,3$           \\ \hline
			$\mathbb Q(\sqrt{-19})$& $(2 \mathcal O_K)^4$	& $2,3$ \\ \hline
			$\mathbb Q(\sqrt{-43})$& $(2 \mathcal O_K)^4$ & $2,3$ \\ \hline
			$\mathbb Q(\sqrt{-67})$& $(2 \mathcal O_K)^4$& $2,3$ \\ \hline
			$\mathbb Q(\sqrt{-163})$& $(2 \mathcal O_K)^4$ & $2,3,5,11,17$\\ \hline 
		\end{tabular}
	\end{table}
\end{center}

Since we have chosen $p \geq 19$ and there is no $p$-torsion in the subgroups of interest, the mod $p$ eigenforms must lift to complex ones.  We obtain a fixed, finite list of cuspidal Bianchi newforms of level dividing $\mathfrak q^4$ to which our mod $p$ eigenform can lift. For each Bianchi newform $\mathfrak f$ in this list, we denote by $\Q_{\ff}$ the number field generated by their eigenvalues. The process described above guarantees that for every $\mathfrak l \nmid \mathfrak q^4 \cdot p$, prime ideal of $K$ we get the following congruence
$$\Tr(\rhobar_{E,p}(\sigma_{\mathfrak l})) \equiv a_{\mathfrak l}(\ff) \pmod{\mathfrak P},$$
between the trace of the image of Frobenius at $\mathfrak l$ in $\rhobar_{E,p}$ and the Hecke eigenvalue of $\ff$ at $\mathfrak l$. Here $\mathfrak P$ is some ideal of $\Q_{\ff}$ that lies above the prime $p$. We now use the idea in \cite{FreitasSiksek}*{Lemma 7.1} to obtain an upper bound on the prime exponent $p$. Although the work in loc. cit. is carried for Hilbert modular forms, the proof of this lemma holds through for Bianchi modular forms $\ff$. We describe the idea below.

Let us fix a prime ideal $\mathfrak l$ as above. The Frey curve $E=E_{a,b,c}$ has good or multiplicative reduction at $\mathfrak l$. If it has good reduction, then $\Tr(\rhobar_{E,p}(\sigma_{\mathfrak l})) \equiv a_{\mathfrak l}(E) \equiv a_{\mathfrak l}(\ff) \pmod{\mathfrak P}$. By definition, $E$ has full two-torsion defined over $K$ and $\mathfrak l \nmid 2$, so $4 \mid \#E(\F_{\mathfrak l})= \norm(\mathfrak l)+1-a_{\mathfrak l}(E)$. Adding the information provided by the Hasse-Weil bounds we get that $a_{\mathfrak l}(E)$ belongs to the finite set
$$\cA_{\mathfrak l} = \{ a \in \Z : |a| \leq 2 \sqrt{\norm(\mathfrak l)}, \, \,  \norm(\mathfrak l) + 1 - a \equiv 0 \pmod{4} \}.$$
If $\fq \nmid 2p$ is a prime of multiplicative reduction, then $$\Tr(\rhobar_{E,p}(\sigma_{\mathfrak l})) = \pm (\norm(\mathfrak l)+1) \Rightarrow \mathfrak P \mid \left(\norm(\mathfrak l)+1 \right)^2-a_{\mathfrak l}(\ff)^2.$$
If $\mathfrak l \mid p$, obviously $p \mid \norm(\mathfrak l)$.
For every prime ideal $\mathfrak l$ of $K$ that does not divide $2$, denote $$B_{\ff,\mathfrak l} = \norm(\mathfrak l) \left( \left( \norm(q)+1 \right)^2 - a_{\fq}(\ff)^2 \right) \prod_{a \in \cA_{\mathfrak l}} (a- a_{\mathfrak l}(\ff) \cdot \cO_{\Q_{\ff}} ).$$
The above proves that $\mathfrak P \mid B_{\ff,\mathfrak l}$ and, by taking norms, that $p \mid \norm(B_{\ff,\mathfrak l})$. 

Using \textbf{Magma}, we computed the cuspidal newforms $\ff$ at levels dividing $\mathfrak q^4$, the fields $\Q_{\ff}$ and eigenvalues $a_{\mathfrak l}(\ff)$ at primes $\mathfrak l$ of $K$ that have norm less than $50$. We computed the ideal $C_{S,\ff} \subseteq \cO_K$, the greatest common divisor of $B_{\ff,\mathfrak l}$ when $\mathfrak l \nmid 2$ runs throw a set $S$ of prime ideals of $K$ that have norm less than $50$. If $C_{S,\ff}$ is not zero, then $p \mid \norm(C_{\ff})$ gives an upper bound on $p$. For every number field $K$ in the statement of Theorem \ref{2inert} and every cuspidal eigenform $\ff$ of level dividing $\mathfrak q^4$, we computed a non-zero constant $C_{S,\ff}$ and we found that $\norm(C_{\ff})$ is only supported on $\{3,5,7 \}$. As $p$ is assumed to be greater or equal to $17$, this is a contradiction and the proof of our theorem is now complete.

\section{Serre's Uniformity Conjecture and Asymptotic Fermat}

For a number field $K$, the \textbf{Asymptotic Fermat's Last Theorem} over $K$ is the statement that there exists a bound $B_K$ such that for all primes $p>B_K$, the Fermat equation $a^p+b^p+c^p=0$ does not have solutions in $a,b,c \in K \setminus \{0 \}$. 

Let $\omega \in \Q(\sqrt{-3})$ be a primitive cube root of unity. For every $p \geq 5$, we have $1^p + \omega^p+ \omega^{2p}=0$, hence the Asymptotic Fermat's Last Theorem does not hold over $\Q(\sqrt{-3})$. The authors of \cite{FrSiKr} point out that it is reasonable to make the following conjecture, a consequence of the $abc$-conjecture for number fields (see \cite{Bro06}).

\begin{conjecture} \label{AFLT}
	Let $K$ be a number field such that $\omega \notin K$. Then the Asymptotic Fermat's Last Theorem holds over $K$.
\end{conjecture}

Theorem \ref{2inert} proved in the previous section can be a little bit unsatisfying, since it only rules out the possible existence of coprime integral solutions $(a,b,c)$ such that $2 \mid abc$. Let $K=\Q(\sqrt{-d})$, where $d=3$, $11$, $19$, $43$, $67$ or $163$, as in the hypothesis of the aforementioned theorem. If $a,b,c \in \mathcal O_K$ is a non-trivial solution to the Fermat equation with prime exponent such that $2 \nmid abc$, then it is explained in Lemmas 4.1 and 4.2 of \cite{AFreitasSiksek} that $E$ has potentially good reduction at $\mathfrak q = 2 \mathcal O_K$ and, after possibly permuting $(a,b,c)$, the Serre conductor of $\rhobar_{E,p}$ is equal to $\mathfrak q^4$.

Assuming Serre's modularity conjecture, we would like to have a full resolution of the Fermat equation
$a^p+b^p+c^p=0$, where  $a,b,c \in K$ and  $p \geq 17$ prime.
Using our approach, this would follow if we could prove that the mod $p$ representation attached to the usual Frey curve is absolutely irreducible. Unfortunately, this is not true. We saw that when $K=\Q(\sqrt{-3})$, the triple formed from the third roots of unity is a solution to the Fermat equation for every prime $p \geq 5$. The Frey curve $E := E_{1, \omega, \omega^2}$ is, for every such $p$, a twist of the CM curve with LMFDB \cite{lmfdb} label 256.1-CMb1. The representation $\rhobar_{E,p}$ of this curve is never absolutely irreducible. To be precise, for $p \geq 5$, the image $\rhobar_{E,p}(G_K)$ is contained in a split Cartan subgroup if $\left( \frac{-3}{p} \right)=1$, respectively in a non-split Cartan subgroup if $\left( \frac{-3}{p}\right)=-1$. The former is reducible whereas the latter is irreducible but absolutely reducible.

To emphasize that a resolution of Fermat equation with prime exponent over the fields $K$ considered above is a task worth pursuing, we will show that such a resolution is possible if we assume a folklore conjecture (see \cite{Bour18}) motivated by a question of Serre.

\begin{conjecture}[Uniformity conjecture] \label{serreuni} Fix a number field $K$. There exists a constant $C(K)$ such that for all non-CM elliptic curves $E/K$ and all primes $p \geq C(K)$, the mod $p$ Galois representation $\rhobar_{E,p}: G_K \to \GL_2(\F_p)$ is surjective.
\end{conjecture}

Let $K$ be one of the six quadratic imaginary fields listed above. Assume that Serre's modularity conjecture and the Uniformity conjecture (Conjectures \ref{serreconj} and \ref{serreuni}) hold for $K$. As there are elliptic curves with $17$-isogenies defined over $K$ (there are such elliptic curves defined over $\Q$), we know that $C(K) > 17$. 

Suppose $(a,b,c) \in K^3$ is a non-trivial solution to the Fermat equation with prime exponent $p \geq C(K)$,
\begin{equation} \label{fermeqp}
a^p+b^p+c^p=0.
\end{equation}

As before, we can scale the solution of \eqref{fermeqp} such that $a,b,c \in \mathcal O_K \setminus \{ 0\}$ are coprime. From Theorem \ref{2inert} we know that $2 \nmid abc$. Let $$E := E_{a,b,c} : Y^2=X(X-a^p)(X+b^p)$$ be the usual Frey curve and denote by $\rhobar_{E,p}: G_K \to \GL_2(\F_p)$ the Galois representation on the $p$-torsion. The discussion in previous sections shows that $\rhobar_{E,p}$ is unramified away from the primes above $2$ and $p$, it is finite flat at every prime of $K$ that lies above $p$ and $\det(\rhobar_{E,p})$ is the mod $p$ cyclotomic character. Lemma 4.2 in \cite{AFreitasSiksek} proves that $E$ has potential good reduction at the prime ideal $2 \cO_K$.  Applying the Tate algorithm one can prove that, after possibly permuting $(a,b,c)$, the valuation of the conductor of $E$ at $2\cO_K$ is equal to $4$ (see \cite{FreitasSiksek}*{Lemma 4.1}).

Suppose that $E$ does not have CM. Our assumption of the Uniformity conjecture implies that $\rhobar_{E,p}$ is surjective, hence absolutely irreducible. Its Serre conductor divides $\cN=2^4\cO_K$. Thus, $\rhobar_{E,p}$ satisfies the hypothesis of Conjecture \ref{serreconj} and this predicts the existence of a mod $p$ eigenform $c \in H^{1}(Y_0(\mathcal N), \overline{\F}_p)$ such that for every prime ideal $(\pi) \subset \mathcal O_K$, coprime to $p \mathcal N$ we have
$$T_{\pi}(c)= \Tr(\rhobar_{E,p}(\Frob_{(\pi)})) \cdot c,$$
where $T_{\pi}$ is a Hecke operator. In the previous section we have shown that this is not possible for $p \geq C(K) > 17$. 

Suppose that $E$ has CM. Serre and Tate \cite{serretate} showed that elliptic curves with CM have everywhere potential good reduction, hence $j(E)$ is an algebraic integer. Recall that
$$j(E)=\frac{c_4(E)^3}{\Delta(E)}=2^{8} \cdot \frac{(b^{2p}-a^pc^p)^3}{(abc)^{2p}}.$$
Since $a$, $b$ and $c$ are integral and coprime, one can see from the usual formulas that $c_4(E)$ and $\Delta(E)$ are coprime outside the prime above $2$. The latter does not divide $abc$. This implies that $abc$ is a unit, since any odd prime dividing $abc$ would be a prime of potentially multiplicative reduction, contradicting the fact that $E$ has CM. Hence $a,b,c$ are units in $\mathcal O_K$.  There are not that many possible units over the six quadratic imaginary fields that are discussed and, by trying all of the possibilities, we obtain that the only solutions are permutations of $(1, \omega, \omega^2) \in \Q(\sqrt{-3})^3$, where $\omega$ is a non-trivial third root of unity. In addition to solutions $(a,b,c) \in K^3$ with $abc=0$ to the Fermat equation \eqref{fermeq}, let us call trivial those with the property that $a+b+c=0$. It can be proved that the latter are just scalar multiples of permutations of $(1,\omega, \omega^2)$. The next result follows by combining the above with Theorems \ref{main} and \ref{2inert}.

\begin{theorem} Let $K=\Q(\sqrt{-d})$ be a quadratic imaginary number field of class number $1$ and suppose that Conjectures \ref{serreconj} and \ref{serreuni} hold over $K$.  There is an absolute constant $C(K)>0$ such that the only solutions to the Fermat equation $a^p+b^p+c^p=0$ with $a,b,c \in K$ and $p>C(K)$ prime are trivial.	
\end{theorem}

\bibliographystyle{amsplain}
\bibliography{perf-pow}

\end{document}